\documentclass[11pt]{amsart}

\usepackage[T1]{fontenc}
\usepackage[latin1]{inputenc}
\usepackage[english]{babel}
\usepackage{amsmath,amssymb,amsthm} 
\usepackage{enumitem}
\usepackage{color}
\usepackage[margin=1in]{geometry}
\usepackage{thmtools}
\usepackage{hyperref}
\usepackage[nameinlink]{cleveref}

\newtheorem{theorem}{Theorem}
\newtheorem{lemma}[theorem]{Lemma}

\theoremstyle{definition}

\newtheorem{remark}[theorem]{Remark}

\newtheorem*{remark*}{Remark}

\newcommand{\diag}{\operatorname{diag}}
\newcommand{\vectwo}[2]{\begin{pmatrix} #1 \\ #2 \end{pmatrix}}
\renewcommand{\MR}[1]{}

\title{On the Cotlar--Stein lemma}
\author[M. Hartz]{Michael Hartz}
\thanks{M.H. was partially supported by the Emmy Noether Program of the German Research Foundation (DFG Grant 466012782)}
\author[M. Scherer]{Marcel Scherer}
\address{Fachrichtung Mathematik, Universit\"at des Saarlandes, 66123 Saarbr\"ucken, Germany}
\email{hartz@math.uni-sb.de}
\address{Technion Israel Institute of Technology, Technion City, Haifa, 3200003, Israel}
\email{scherer@math.uni-sb.de}
\date{\today}

\subjclass[2020]{Primary: 47A63; Secondary 47B90}
\keywords{Cotlar--Stein lemma, almost orthogonality, Schur test}

\begin{document}

\begin{abstract}
  We give a direct proof of the Cotlar--Stein lemma, which does not rely on the power trick.
\end{abstract}

\maketitle

The goal of this note is to give a direct proof of the following result.

\begin{theorem}[Cotlar--Stein lemma]
  \label{thm:cotlar-stein}
  Let $T_1,\ldots,T_n$ be bounded linear operators on a Hilbert space. Then
  \begin{equation*}
    \Big\| \sum_{k=1}^n T_k \Big\|^2 \le \Big( \max_{j=1,\ldots,n} \sum_{k=1}^n \sqrt{\|T_j^* T_k\|} \Big) \Big( \max_{j=1,\ldots,n} \sum_{k=1}^n \sqrt{\|T_j T_k^*\|} \Big).
  \end{equation*}
\end{theorem}

This result was first proved by Cotlar \cite{Cotlar55} for commuting self-adjoint operators and later independently by Cotlar and by Knapp--Stein in the general case; see \cite{KS69,KS71}.
It is frequently used in harmonic analysis to establish $L^2$-boundedness of singular integral operators and pseudodifferential operators
by decomposing them into sums of simpler operators which are almost orthogonal in the sense that the off-diagonal terms $\|T_j^* T_k\|$ and $\|T_j T_k^*\|$ for $j \neq k$ are small.
See for instance \cite[Chapter VII]{Stein93} or \cite[Section 4.5]{Grafakos14}.
A very nice discussion of the Cotlar--Stein lemma can be found in the blog post by Terence Tao \cite{Tao11}.

As explained by Tao \cite{Tao11}, the Cotlar--Stein lemma can be regarded as a non-commutative version of the Schur test for matrices $(a_{jk}) \in M_n(\mathbb{C})$, which states that
\begin{equation}
  \label{eq:schur_test}
  \|(a_{jk})\|^2 \le \Big(\max_{j=1,\ldots,n} \sum_{k=1}^n |a_{jk}| \Big) \Big( {\max_{j=1,\ldots,n} \sum_{k=1}^n |a_{kj}|} \Big).
\end{equation}
For instance, just as the Schur test generalizes the basic operator norm equality $\|\diag(d_1,\ldots,d_n)\| = \max_{j=1,\ldots,n} |d_j|$ for diagonal matrices, the Cotlar--Stein lemma generalizes the basic operator norm equality $\|\sum_{k=1}^n T_k\| = \max_{k=1,\ldots,n} \|T_k\|$ for operators
$T_1,\ldots,T_n$ with pairwise orthogonal ranges (i.e.\ $T_j^* T_k = 0$ for $j \neq k$) and pairwise orthogonal adjoint ranges (i.e.\ $T_j T_k^* = 0$ for $j \neq k$).

This connection becomes quite explicit in light of the following sharpening of \Cref{thm:cotlar-stein} due to Calder\'on and Vaillancourt \cite{CV71,CV72}.

\begin{theorem}
  \label{thm:improved_cotlar_stein}
  Let $T_1,\ldots,T_n$ be bounded linear operators on a Hilbert space. Let $a_{jk} = \sqrt{\|T_j T_k^*\|}$ and $b_{jk} = \sqrt{\|T_j^* T_k\|}$ for $j,k=1,\ldots,n$.
  Then
  \begin{equation*}
    \Big\| \sum_{k=1}^n T_k \Big\|^2 \le \|(a_{jk})\|  \,\|(b_{jk})\|.
  \end{equation*}
\end{theorem}

\Cref{thm:cotlar-stein} immediately follows from \Cref{thm:improved_cotlar_stein} and the Schur test \eqref{eq:schur_test}
applied to the real symmetric matrices $(a_{jk})$ and $(b_{jk})$.
Simple examples in which the operators are scalars (e.g.\ $T_1 =1$ and $T_2= \ldots = T_n = \frac{1}{n}$)
show that the right-hand side of \Cref{thm:cotlar-stein} can exceed the right-hand side of \Cref{thm:improved_cotlar_stein} by a factor on the order of $n$.
   
The usual proof of \Cref{thm:cotlar-stein} (and of \Cref{thm:improved_cotlar_stein}) relies on the so-called power trick.
The basic idea is that writing $S = \sum_{k=1}^n T_k$ and $M$ for the right-hand side in \Cref{thm:cotlar-stein}, one shows that for all $m \in \mathbb N$,
\[
    \|S\|^{2 m} = \| (S S^*)^m\| \le n M^m.
\]
The desired bound is then obtained by taking $m$-th roots and letting $m \to \infty$, which gets rid
of the factor of $n$.
Our proof instead relies on arguments involving operator matrices.
The necessary ingredients already exist in the literature. In particular, the case where the operators
in \Cref{thm:improved_cotlar_stein} are all positive follows from an inequality of Popovici and Sebesty\'en \cite{PS06}; the general case can be deduced from this case with the help of polar decomposition.

The goal of this note is to explain these ingredients and their relation to the Cotlar--Stein lemma.
We start with a standard inequality.

\begin{lemma}
  \label{lem:ineq_1}
  Let $A,B$ be positive operators on a Hilbert space. Then $\|A^{1/2} B^{1/2}\| \le \|A B\|^{1/2}$.
\end{lemma}

\begin{proof}
  We use the $C^*$-identity, the agreement of spectral radius $r$ and norm for self-adjoint operators, and the fact that
  $r(XY) = r(YX)$ for any two operators $X,Y$ to find that
  \begin{equation*}
    \|A^{1/2} B^{1/2}\|^2 = \|B^{1/2} A B^{1/2}\| = r(B^{1/2} A B^{1/2}) = r(A B) \le \|A B\|. \qedhere
  \end{equation*}
\end{proof}

The following inequality, which is \Cref{thm:improved_cotlar_stein} for positive (but not necessarily commuting) operators, was essentially shown by Popovici and Sebesty\'en
\cite[Remark 2.4]{PS06}; we provide the short argument.
\begin{lemma}
  \label{lem:positive}
  For positive operators $A_1,\ldots,A_n$ on $\mathcal H$, we have
  \begin{equation*}
    \Big\| \sum_{k=1}^n A_k \Big\| \le \big\| ( \|A_j A_k\|^{1/2} ) \big\|.
  \end{equation*}
\end{lemma}

\begin{proof}
  Let
  \[
  R =
  \begin{bmatrix}
    A_1^{1/2} & A_2^{1/2} & \ldots & A_n^{1/2}
  \end{bmatrix}: \mathcal H^n \to \mathcal H, \quad (x_1,\ldots,x_n) \mapsto \sum_{k=1}^n A_k^{1/2} x_k,
  \]
  be the row operator. One can calculate with operator matrices as one does with block matrices (cf.\ \cite[Chapter 8]{Halmos82}), so
  \begin{equation*}
    \Big\| \sum_{k=1}^n A_k \Big\| = \|R R^*\| = \|R^* R\| = \big\| ( A_j^{1/2} A_k^{1/2} ) \big\|
    \le \big\| ( \|A_j^{1/2} A_k^{1/2}\| ) \big\|,
  \end{equation*}
  where the last step uses the elementary norm comparison of an operator matrix
  with the scalar matrix formed from the norms of its entries.
  The proof is completed by an application of \Cref{lem:ineq_1}.
\end{proof}

To extend the result to general operators, we use polar decomposition. We write $|T| = (T^* T)^{1/2}$.

\begin{lemma}
  \label{lem:absolute_value_bound}
  Let $T_1,\ldots,T_n$ be bounded linear operators on a Hilbert space. Then
  \begin{equation*}
    \Big\| \sum_{k=1}^n T_k \Big\|^2 \le \Big\| \sum_{k=1}^n |T_k| \Big\| \Big\| \sum_{k=1}^n |T_k^*| \Big\|.
  \end{equation*}
\end{lemma}

\begin{proof}
  Let $T$ be any operator with polar decomposition $T = V |T|$, where $V$ is a partial isometry.
 Then $|T^*| = V |T| V^*$ and so
  \begin{equation*}
    \begin{bmatrix}
      |T| & T^* \\
      T & |T^*|
    \end{bmatrix}
    =
    \begin{bmatrix}
      |T|^{1/2} \\ V |T|^{1/2}
    \end{bmatrix}
    \begin{bmatrix}
      |T|^{1/2} & |T|^{1/2} V^*
    \end{bmatrix}
  \end{equation*}
  is positive. Hence so is
  \begin{equation*}
  P:=
    \begin{bmatrix}
      \sum_{k=1}^n |T_k| & \sum_{k=1}^n T_k^* \\
      \sum_{k=1}^n T_k & \sum_{k=1}^n |T_k^*|
    \end{bmatrix}.
  \end{equation*}
  The Cauchy--Schwarz inequality for this positive operator matrix, namely
  \[
  \Big| \Big\langle P \vectwo{x}{0}, \vectwo{0}{y} \Big\rangle \Big|^2 \le \Big\langle P \vectwo{x}{0}, \vectwo{x}{0} \Big\rangle \Big\langle P \vectwo{0}{y}, \vectwo{0}{y} \Big\rangle,
  \]
  implies that
  \begin{equation}
  \label{eqn:mixed_schwarz}
    \Big|\Big\langle \sum_{k=1}^n T_k x,y  \Big\rangle \Big|^2 \le  \Big\langle \sum_{k=1}^n |T_k| x,x \Big\rangle \Big\langle \sum_{k=1}^n |T_k^*| y,y \Big\rangle
  \end{equation}
  for all vectors $x,y$, which yields the lemma.
\end{proof}

We need one last lemma.

\begin{lemma}
  \label{lem:norm_abs_value}
  Let $T,S$ be bounded linear operators on a Hilbert space. Then $\| |T| |S| \| = \|T S^*\|$.
\end{lemma}

\begin{proof}
  This is a repeated application of the $C^*$-identity:
  \begin{equation*}
    \| |T| |S| \|^2 = \| |S| T^* T |S| \| = \| T |S|^2 T^* \| = \| T S^* S T^* \| = \| T S^* \|^2. \qedhere
  \end{equation*}
  
\end{proof}

We are now ready to prove \Cref{thm:improved_cotlar_stein}.

\begin{proof}[Proof of \Cref{thm:improved_cotlar_stein}]
  Let $A_j = |T_j|$ and $B_j = |T_j^*|$ for $j=1,\ldots,n$. Then $A_j$ and $B_j$ are positive operators and $\|A_j A_k\|^{1/2} = \|T_j T_k^*\|^{1/2} = a_{jk}$ and $\|B_j B_k\|^{1/2} = \|T_j^* T_k\|^{1/2} = b_{jk}$ for $j,k=1,\ldots,n$ by \Cref{lem:norm_abs_value}.
  Hence \Cref{lem:absolute_value_bound} and \Cref{lem:positive} imply that
  \begin{equation*}
    \Big\| \sum_{k=1}^n T_k \Big\|^2 \le \Big\| \sum_{k=1}^n A_k \Big\| \Big\| \sum_{k=1}^n B_k \Big\| \le \|(a_{jk})\|  \,\|(b_{jk})\|. \qedhere
  \end{equation*}
\end{proof}

\begin{remark}
  \begin{enumerate}
\item The proof of \Cref{thm:improved_cotlar_stein} gives the smaller, but somewhat more complicated bound
  \begin{equation*}
    \Big\| \sum_{k=1}^n T_k \Big\|^2 \le \Big\| \big( \| |T_j|^{1/2} |T_k|^{1/2} \| \big) \Big\|  \,\Big\| \big( \| |T_j^*|^{1/2} |T_k^*|^{1/2} \| \big) \Big\|.
  \end{equation*}
  \item The Cotlar--Stein lemma extends to infinite sums $\sum_{k=1}^\infty T_k$; the same is true for \Cref{thm:improved_cotlar_stein}: If $(a_{jk})$ and $(b_{jk})$ are bounded operators on $\ell^2$, then the sum $\sum_{k=1}^\infty T_k$ converges in the strong operator topology to an operator of norm at
  most $C:=\sqrt{\| (a_{jk})\| \|(b_{jk})\|}$.
  This can be seen from our proof as follows.
  
  For each $n$, \Cref{thm:improved_cotlar_stein} implies that $\|\sum_{k=1}^n T_k\| \le C$,
  so it suffices to establish convergence. \Cref{lem:positive} and \Cref{lem:norm_abs_value} show that
  $\| \sum_{k=1}^n |T_k| \| \| \sum_{k=1}^n |T_k^*| \| \le C^2$,
  hence the two factors are individually bounded since they are increasing in $n$.
  In particular, $\sum_{k=1}^\infty \langle |T_k| x, x \rangle < \infty$
  for each vector $x$.
  Now inequality
  \eqref{eqn:mixed_schwarz} implies that
  \[
    \Big\| \sum_{k=m}^n T_k x \Big\|^2 \le  \sum_{k=m}^n \langle |T_k| x, x \rangle \, \Big\| \sum_{k=m}^n |T_k^*| \Big\|,
  \]
  hence $(\sum_{k=1}^n T_k x)$ is a Cauchy sequence.

  Alternatively, one can use \Cref{thm:improved_cotlar_stein}
  and the following general Hilbert space lemma: If $\| \sum_k \epsilon_k T_k \| \le C$ for every finitely supported sequence $(\epsilon_k)$ of modulus at most $1$, then $\sum_k T_k$ converges in the strong operator topology
  to an operator of norm at most $C$;
  see \cite[p.\ 318]{Stein93} or the proof of \cite[Lemma 4.5.1]{Grafakos14} and also \cite[Section 2.4]{AK06} for the broader Banach space context.
\item The key inequality \eqref{eqn:mixed_schwarz} in the proof of \Cref{lem:absolute_value_bound}, at least for a single operator, is sometimes known as the mixed Schwarz inequality, cf.\ \cite[Problem 138]{Halmos82} and \cite[p. 420]{Heinz51}. A result more general than \Cref{lem:absolute_value_bound} was proved by Kittaneh \cite[Theorem 2]{Kittaneh99}.
\end{enumerate}
\end{remark}

\textbf{AI Disclosure:} ChatGPT was used to perform literature search and preliminary computations
and for proofreading.
Github Copilot was used for autocomplete. Aside from this, the note is human-generated.

\bibliographystyle{amsplain}
\bibliography{cotlar-stein}

@Article{Cotlar55,
  author     = {Cotlar, M.},
  journal    = {Rev. Mat. Cuyana},
  title      = {A combinatorial inequality and its applications to {$L\sp 2$}-spaces},
  year       = {1955},
  issn       = {0484-7822},
  pages      = {41--55 (1956)},
  volume     = {1},
  fjournal   = {Revista Matem\'{a}tica Cuyana},
  mrclass    = {46.1X},
  mrnumber   = {80263},
  mrreviewer = {Edwin\ Hewitt},
}

@Article{KS71,
  author     = {Knapp, A. W. and Stein, E. M.},
  journal    = {Ann. of Math. (2)},
  title      = {Intertwining operators for semisimple groups},
  year       = {1971},
  issn       = {0003-486X},
  pages      = {489--578},
  volume     = {93},
  doi        = {10.2307/1970887},
  fjournal   = {Annals of Mathematics. Second Series},
  mrclass    = {22E45},
  mrnumber   = {460543},
  mrreviewer = {G.\ I.\ Ol\cprime shanski\u{\i}},
  url        = {https://doi.org/10.2307/1970887},
}

@Article{Kittaneh99,
  author     = {Fuad Kittaneh},
  journal    = {Canadian Mathematical Bulletin},
  title      = {Some Norm Inequalities for Operators},
  volume     = {42},
  number     = {1},
  pages      = {87--96},
  year       = {1999},
  doi        = {10.4153/CMB-1999-010-6}
}

@Article{KS69,
  author     = {Knapp, A. W. and Stein, E. M.},
  journal    = {Proc. Nat. Acad. Sci. U.S.A.},
  title      = {Singular integrals and the principal series. {I}, {II}},
  year       = {1969},
  issn       = {0027-8424},
  pages      = {281--284; ibid. 66 (1969), 13--17},
  volume     = {63},
  doi        = {10.1073/pnas.66.1.13},
  fjournal   = {Proceedings of the National Academy of Sciences of the United States of America},
  mrclass    = {22.60},
  mrnumber   = {263989},
  mrreviewer = {G.\ Warner},
  url        = {https://doi.org/10.1073/pnas.66.1.13},
}

@Book{Stein93,
  author     = {Stein, Elias M.},
  publisher  = {Princeton University Press, Princeton, NJ},
  title      = {Harmonic analysis: real-variable methods, orthogonality, and oscillatory integrals},
  year       = {1993},
  isbn       = {0-691-03216-5},
  note       = {With the assistance of Timothy S. Murphy, Monographs in Harmonic Analysis, III},
  series     = {Princeton Mathematical Series},
  volume     = {43},
  mrclass    = {42-02 (35Sxx 43-02 47G30)},
  mrnumber   = {1232192},
  mrreviewer = {Michael\ Cowling},
  pages      = {xiv+695},
}

@Unpublished{Tao11,
  author = {Terence Tao},
  note   = {https://terrytao.wordpress.com/2011/05/25/the-cotlar-stein-lemma/},
  title  = {The {C}otlar-{S}tein lemma},
  year   = {2011},
}

@Book{Grafakos14,
  author     = {Grafakos, Loukas},
  publisher  = {Springer, New York},
  title      = {Modern {F}ourier analysis},
  year       = {2014},
  edition    = {Third},
  isbn       = {978-1-4939-1229-2; 978-1-4939-1230-8},
  series     = {Graduate Texts in Mathematics},
  volume     = {250},
  doi        = {10.1007/978-1-4939-1230-8},
  mrclass    = {42-01 (42Bxx)},
  mrnumber   = {3243741},
  mrreviewer = {Atanas G. Stefanov},
  pages      = {xvi+624},
}

@Article{PS06,
  author     = {Popovici, Dan and Sebesty\'{e}n, Zolt\'{a}n},
  journal    = {J. Operator Theory},
  title      = {Norm estimations for finite sums of positive operators},
  year       = {2006},
  issn       = {0379-4024,1841-7744},
  number     = {1},
  pages      = {3--15},
  volume     = {56},
  fjournal   = {Journal of Operator Theory},
  mrclass    = {47B65 (47A30)},
  mrnumber   = {2261609},
  mrreviewer = {Jin-Chuan\ Hou},
}

@article {Heinz51,
    AUTHOR = {Heinz, Erhard},
     TITLE = {Beitr\"{a}ge zur {S}t\"{o}rungstheorie der
              {S}pektralzerlegung},
   JOURNAL = {Math. Ann.},
  FJOURNAL = {Mathematische Annalen},
    VOLUME = {123},
      YEAR = {1951},
     PAGES = {415--438},
      ISSN = {0025-5831,1432-1807},
   MRCLASS = {46.3X},
  MRNUMBER = {44747},
MRREVIEWER = {F.\ H.\ Brownell},
       DOI = {10.1007/BF02054965},
       URL = {https://doi.org/10.1007/BF02054965},
}

@Book{Halmos82,
  Title                    = {A {H}ilbert space problem book},
  Author                   = {Halmos, Paul Richard},
  Publisher                = {Springer-Verlag},
  Year                     = {1982},

  Address                  = {New York},
  Edition                  = {Second},
  Note                     = {Encyclopedia of Mathematics and its Applications, 17},
  Series                   = {Graduate Texts in Mathematics},
  Volume                   = {19},
  ISBN                     = {0-387-90685-1},
  Mrclass                  = {47-01 (46-01)},
  Mrnumber                 = {675952 (84e:47001)},
  Mrreviewer               = {J. Weidmann},
  Pages                    = {xvii+369}
}

@Book{AK06,
  Title                    = {Topics in {B}anach space theory},
  Author                   = {Albiac, Fernando and Kalton, Nigel J.},
  Publisher                = {Springer},
  Year                     = {2006},

  Address                  = {New York},
  Series                   = {Graduate Texts in Mathematics},
  Volume                   = {233},
  ISBN                     = {978-0387-28141-4; 0-387-28141-X},
  Mrclass                  = {46B20 (46-01)},
  Mrnumber                 = {2192298 (2006h:46005)},
  Mrreviewer               = {Gilles Godefroy},
  Pages                    = {xii+373}
}

@Article{CV71,
  author     = {Calder\'on, Alberto-P. and Vaillancourt, R\'emi},
  journal    = {J. Math. Soc. Japan},
  title      = {On the boundedness of pseudo-differential operators},
  year       = {1971},
  issn       = {0025-5645,1881-1167},
  pages      = {374--378},
  volume     = {23},
  doi        = {10.2969/jmsj/02320374},
  fjournal   = {Journal of the Mathematical Society of Japan},
  mrclass    = {47.70 (35.00)},
  mrnumber   = {284872},
  mrreviewer = {F.\ Cardoso},
}

@Article{CV72,
  author     = {Calder\'on, Alberto-P. and Vaillancourt, R\'emi},
  journal    = {Proc. Nat. Acad. Sci. U.S.A.},
  title      = {A class of bounded pseudo-differential operators},
  year       = {1972},
  issn       = {0027-8424},
  pages      = {1185--1187},
  volume     = {69},
  doi        = {10.1073/pnas.69.5.1185},
  fjournal   = {Proceedings of the National Academy of Sciences of the United States of America},
  mrclass    = {47G05 (35S05)},
  mrnumber   = {298480},
  mrreviewer = {F.\ Cardoso},
}

\end{document}